\documentclass[final,leqno]{siamltex}

\usepackage{amsmath,amssymb}
\usepackage{graphicx}
\usepackage{pgfplots}

\newtheorem{remark}{Remark}

\title{SCALABLE AND FAULT TOLERANT COMPUTATION WITH THE 
SPARSE GRID COMBINATION TECHNIQUE\thanks{This research was 
supported by the Australian Research Council's {\em Linkage Projects} 
funding scheme (project number LP110200410). We are grateful to 
Fujitsu Laboratories of Europe for providing funding as the collaborative 
partner in this project.}}

\author{Brendan Harding\footnotemark[2] 
\and Markus Hegland\footnotemark[2] 
\and Jay Larson\footnotemark[2] 
\and James Southern\footnotemark[3]}

\begin{document}

\maketitle

\renewcommand{\thefootnote}{\fnsymbol{footnote}}

\footnotetext[2]{Mathematical Sciences Institute, Australian National University, Canberra, Australian Capital Territory, Australia, 0200.}
\footnotetext[3]{Fujitsu Laboratories of Europe, Hayes Park Central, Hayes End Road, Hayes, Middlesex, UB4 8FE, United Kingdom.}

\renewcommand{\thefootnote}{\arabic{footnote}}

\begin{abstract}
This paper continues to develop a fault tolerant extension of the sparse grid 
combination technique recently proposed in~{[{B.~Harding and M.~Hegland}, 
{\em ANZIAM J.}, 54~(CTAC2012), pp.~C394--C411]}. 
The approach is novel for two reasons, first it provides several levels in which
one can exploit parallelism leading towards massively parallel implementations, 
and second, it provides algorithm-based fault tolerance so that solutions can 
still be recovered if failures occur during computation. 
We present a generalisation of the combination technique from which the 
fault tolerant algorithm is a consequence. 
Using a model for the time between faults on each node of a high performance 
computer we provide bounds on the expected error for interpolation with this 
algorithm. 
Numerical experiments on the scalar advection PDE demonstrate that the
algorithm is resilient to faults on a real application. 
It is observed that the trade-off of recovery time to decreased accuracy 
of the solution is suitably small. 
A comparison with traditional checkpoint-restart methods applied to the 
combination technique show that our approach is highly scalable with
respect to the number of faults.
\end{abstract}

\begin{keywords} 
exascale computing, algorithm-based fault tolerance, sparse grid combination technique, parallel algorithms
\end{keywords}

\begin{AMS}
65Y05, 68W10
\end{AMS}

\pagestyle{myheadings}
\thispagestyle{plain}
\markboth{B. HARDING, M. HEGLAND, J. LARSON AND J. SOUTHERN}{SCALABLE AND FAULT TOLERANT COMBINATION TECHNIQUE}

\section{Introduction}

Many recent survey articles on the challenges of achieving exascale computing
identify three issues to be overcome: exploiting massive parallelism,
reducing energy usage and, in particular, coping with run-time
failures~\cite{cappello,gsd,cggkks}.
Faults are an issue at peta/exa-scale due to the increasing number of components in
such systems. 
Traditional checkpoint-restart based solutions become unfeasible at this scale as the 
decreasing mean time between failures approaches the time required to checkpoint 
and restart an application. 
Algorithm based fault tolerance has been studied as a promising solution to this
issue for many problems~\cite{huang_abraham,bddl}.

Sparse grids were introduced in the study of high dimensional problems as a way
to reduce the {\em curse of dimensionality}.
They are based on the observation that when a solution on a regular grid is 
decomposed into its hierarchical bases the highest frequency components 
contribute the least to sufficiently smooth solutions.
Removing some of these high frequency components has a small impact on
the accuracy of the solution whilst significantly reducing the computational 
complexity~\cite{griebel,griebel_bungartz}.
The combination technique was introduced to approximate sparse grid solutions
without the complications of computing with a hierarchical basis.
In recent years these approaches have been applied to a wide variety of
applications from real time visualisation of complex datasets to solving high
dimensional problems that were previously cumbersome~\cite{murarasu,griebel}.

Previously~\cite{my_ctac_paper,jay_parco,my_sga_paper} it has been described how
the combination technique can be implemented within a {\em Map Reduce} framework.
Doing so allows one to exploit an extra layer of parallelism and fault tolerance
can be achieved by recomputing failed map tasks as described in~\cite{dean}.
Also proposed was an alternative approach to fault tolerance in which
recomputation can be avoided for a small trade off in solution error.
In~\cite{my_parco_paper} we demonstrated this approach for a simple two-dimensional 
problem showing that the average solution error after simulated
faults was generally close to that without faults.
In this paper we develop and discuss this approach in much greater detail. In
particular we develop a general theory for computing new combination
coefficients and discuss a three-dimensional implementation based on MPI and OpenMP
which scales well for relatively small problems.
As has been done in the previous 
literature~\cite{my_ctac_paper,jay_parco,my_sga_paper}, 
we use the solution of the scalar advection PDE for our numerical experiments.

The remainder of the paper is organised as follows. 
In Section~\ref{sec:back} we review the combination technique and provide 
some well-known results which are relevant to our analysis of the fault 
tolerant combination technique. 
We then develop the notion of a general combination technique. 

In Section~\ref{sec:faults} we describe how the combination technique can be
modified to be fault tolerant as an application of the general combination technique. 
Using a simple model for faults on each node of a supercomputer we are able 
to model the failure of component grids in the combination technique and apply 
this to the simulation of faults in our code. 
We present bounds on the expected error and discuss in how faults affect the
scalability of the algorithm as a whole.

In Section~\ref{sec:implem} we describe the details of our implementation. 
In particular we discuss the multi-layered approach and the way in which
multiple components work together in order to harness the many levels of
parallelism. We also discuss the scalability bottleneck caused by communications 
and several ways in which one may address this.

Finally, in Section~\ref{sec:numres} we present numerical results obtained by 
running our implementation with simulated faults on a PDE solver. We demonstrate that our 
approach scales well to a large number of faults and has a relatively small impact on 
the solution error.

\section{The Combination Technique and a Generalisation}\label{sec:back}

We introduce the combination technique and a classical result which will be used in our analysis of the fault tolerant algorithm. For a complete introduction of the combination technique one should
refer to~\cite{garcke,griebel,griebel_bungartz}. 
We then go on to extend this to a more general notion of a combination technique building on existing work on adaptive sparse grids~\cite{hegland_asg}.

\subsection{The Combination Technique}\label{sec:cct}

Let $i\in\mathbb{N}$, then we define $\Omega_{i}:=\{k 2^{-i}:k=0,\dots,2^{i}\}$
to be a discretisation of the unit interval. Similarly for $i\in\mathbb{N}^{d}$
we define $\Omega_{i}:=\Omega_{i_{1}}\times\cdots\times\Omega_{i_{d}}$ as a grid
on the unit $d$-cube. Throughout the rest of this paper we treat the variables
$i,j$ as multi-indices in $\mathbb{N}^{d}$. We say $i\leq j$ if and only if
$i_{k}\leq j_{k}$ for all $k\in\{1,\dots,d\}$, and similarly, $i<j$ if and only
if $i\leq j$ and $i\neq j$.

Now suppose we have a problem with solution $u\in V\subset C([0,1]^{d})$, then
we use $V_{i}\subset V$ to denote the function space consisting of piecewise
linear functions uniquely determined by their values on the grid $\Omega_{i}$.
Further, we denote an approximation of $u$ in the space $V_{i}$ by $u_{i}$.
The sparse grid space of level $n$ is defined to be
$V^{s}_{n}:=\sum_{\|i\|_{1}\leq n}V_{i}$. A sparse grid solution is a
$u^{s}_{n}\in V^{s}_{n}$ which closely approximates $u\in V$.
The combination technique approximates a sparse grid solution by taking the sum
of several solutions from different anisotropic grids. The classical combination
technique is given by the equation
\begin{equation}\label{eqn:cct}
u^{c}_{n}:=\sum_{k=0}^{d-1}(-1)^{k}\binom{d-1}{k}\sum_{\|i\|_{1}=n-k}u_{i} \,.
\end{equation}
Fundamentally, this is an application of the inclusion{\slash}exclusion
principle. This can be seen if the function spaces are viewed as a
lattice~\cite{hegland_asg}. For example, if one wishes to add the functions
$u_{i}\in V_{i}$ and $u_{j}\in V_{j}$ then the result will have two
contributions from the intersection space $V_{i\wedge j}=V_{i}\cap V_{j}$, with
$i\wedge j=(\min\{i_{1},j_{1}\},\dots,\min\{i_{d},j_{d}\})$. To avoid this we
simply take $u_{i}+u_{j}-u_{i\wedge j}$. This can be seen in
Figure~\ref{fig:ct2d}, which shows a level 4 combination in 2 dimensions.

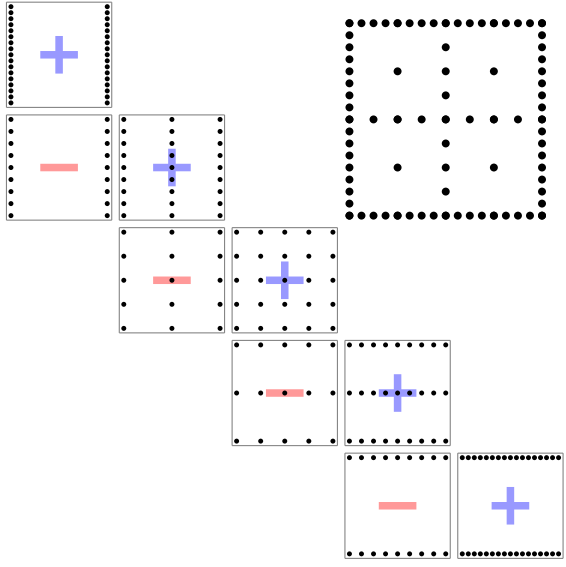
\begin{figure}
\centering
\begin{tikzpicture}
\scriptsize
\foreach \i in {0,...,4} {
	\pgfmathtruncatemacro{\x}{\i};
	\draw[black!50] (-0.06+1.5*\x, 5.94-1.5*\x) -- (-0.06+1.5*\x+1.4, 5.94-1.5*\x) {};
	\draw[black!50] (-0.06+1.5*\x, 5.94-1.5*\x) -- (-0.06+1.5*\x, 5.94-1.5*\x+1.4) {};
	\draw[black!50] (-0.06+1.5*\x+1.4, 5.94-1.5*\x) -- (-0.06+1.5*\x+1.4, 5.94-1.5*\x+1.4) {};
	\draw[black!50] (-0.06+1.5*\x, 5.94-1.5*\x+1.4) -- (-0.06+1.5*\x+1.4, 5.94-1.5*\x+1.4) {};
	\draw[blue!40,line width=1.0mm] (0.64+1.5*\x-0.25, 6.0+0.64-1.5*\x) -- (0.64+1.5*\x+0.25, 6.0+0.64-1.5*\x);
	\draw[blue!40,line width=1.0mm] (0.64+1.5*\x, 6.0+0.64-1.5*\x-0.25) -- (0.64+1.5*\x, 6.0+0.64-1.5*\x+0.25);
}
\foreach \i in {0,...,3} {
	\pgfmathtruncatemacro{\x}{\i};
	\draw[black!50] (-0.06+1.5*\x, 4.44-1.5*\x) -- (-0.06+1.5*\x+1.4, 4.44-1.5*\x) {};
	\draw[black!50] (-0.06+1.5*\x, 4.44-1.5*\x) -- (-0.06+1.5*\x, 4.44-1.5*\x+1.4) {};
	\draw[black!50] (-0.06+1.5*\x+1.4, 4.44-1.5*\x) -- (-0.06+1.5*\x+1.4, 4.44-1.5*\x+1.4) {};
	\draw[black!50] (-0.06+1.5*\x, 4.44-1.5*\x+1.4) -- (-0.06+1.5*\x+1.4, 4.44-1.5*\x+1.4) {};
	\draw[red!40,line width=1.0mm] (0.64+1.5*\x-0.25, 4.5+0.64-1.5*\x) -- (0.64+1.5*\x+0.25, 4.5+0.64-1.5*\x);
}
\foreach \i in {0,...,33} { 
	\pgfmathtruncatemacro{\y}{\i / 2};
	\pgfmathtruncatemacro{\x}{\i - 2 * \y};
	\node[fill,circle,scale=0.25] at (0.0+1.28*\x,6.0+0.08*\y) {};
	\node[fill,circle,scale=0.25] at (6.0+0.08*\y,0.0+1.28*\x) {};
}
\foreach \i in {0,...,26} { 
	\pgfmathtruncatemacro{\y}{\i / 3};
	\pgfmathtruncatemacro{\x}{\i - 3 * \y};
	\node[fill,circle,scale=0.25] at (1.5+0.64*\x,4.5+0.16*\y) {};
	\node[fill,circle,scale=0.25] at (4.5+0.16*\y,1.5+0.64*\x) {};
}
\foreach \i in {0,...,24} { 
	\pgfmathtruncatemacro{\y}{\i / 5};
	\pgfmathtruncatemacro{\x}{\i - 5 * \y};
	\node[fill,circle,scale=0.25] at (3.0+0.32*\x,3.0+0.32*\y) {};
}
\foreach \i in {0,...,17} { 
	\pgfmathtruncatemacro{\y}{\i / 2};
	\pgfmathtruncatemacro{\x}{\i - 2 * \y};
	\node[fill,circle,scale=0.25] at (0.0+1.28*\x,4.5+0.16*\y) {};
	\node[fill,circle,scale=0.25] at (4.5+0.16*\y,0.0+1.28*\x) {};
}
\foreach \i in {0,...,14} { 
	\pgfmathtruncatemacro{\y}{\i / 3};
	\pgfmathtruncatemacro{\x}{\i - 3 * \y};
	\node[fill,circle,scale=0.25] at (1.5+0.64*\x,3.0+0.32*\y) {};
	\node[fill,circle,scale=0.25] at (3.0+0.32*\y,1.5+0.64*\x) {};
}
\foreach \i in {0,...,33} { 
	\pgfmathtruncatemacro{\y}{\i / 2};
	\pgfmathtruncatemacro{\x}{\i - 2 * \y};
	\node[fill,circle,scale=0.4] at (4.5+2.56*\x,4.5+0.16*\y) {};
	\node[fill,circle,scale=0.4] at (4.5+0.16*\y,4.5+2.56*\x) {};
}
\foreach \i in {0,...,8} { 
	\pgfmathtruncatemacro{\y}{\i};
	\pgfmathtruncatemacro{\x}{\i};
	\node[fill,circle,scale=0.4] at (4.5+1.28,4.5+0.32*\y) {};
	\node[fill,circle,scale=0.4] at (4.5+0.32*\x,4.5+1.28) {};
}
\foreach \i in {0,...,9} { 
	\pgfmathtruncatemacro{\y}{\i / 2};
	\pgfmathtruncatemacro{\x}{\i - 2 * \y};
	\node[fill,circle,scale=0.4] at (4.5+0.64+1.28*\x,4.5+0.64*\y) {};
}
\end{tikzpicture}
\caption{\label{fig:ct2d} A level 4 combination in 2 dimensions. The 9 component grids are arranged according to the frequency of data points in each dimension. A plus or minus denotes a combination coefficient of $+1$ or $-1$ respectively. In the top right is the (enlarged) sparse grid corresponding to the union of the component grids.}
\end{figure}

An important concept in the development of sparse grids is that of the
hierarchical space (sometimes also referred to as a hierarchical surplus). A
simple definition of the hierarchical space $W_{i}$ is the space of all
functions $f_{i}\in V_{i}$ such that $f_{i}$ is zero when sampled on all grid
points in the set $\bigcup_{j<i}\Omega_{j}$. Equivalently we have
$V_{i}=W_{i}\oplus\sum_{j<i}V_{j}$. Noting that $V_{i}=\bigoplus_{j\leq
i}W_{j}$, a hierarchical decomposition of $u_{i}\in V_{i}$ is the computation of
the unique components $h_{j}\in W_{j}$ for $j\leq i$ such that
$u_{i}=\sum_{j\leq i}h_{j}$. The sparse grid space can also be written in terms
of hierarchical spaces as $V^{s}_{n}=\bigoplus_{\|i\|_{1}\leq n}W_{i}$.

Let $H^{2}_{\text{mix}}$ be the Sobolev-space with dominating mixed derivatives
with norm
\begin{equation}
\|u\|^{2}_{H^{2}_{\text{mix}}}=\sum_{\|i\|_{\infty}\leq 2}\left\|\frac{\partial^{\|i\|_{1}}}{\partial x^{i}}u\right\|^{2}_{2}
\end{equation}
If $u\in H^{2}_{\text{mix}}$ then we have the estimate $\|h_{j}\|_{2}\leq
3^{-d}2^{-2\|i\|_{1}}|u|_{H^{2}_{\text{mix}}}$ for each of the hierarchical spaces where
$|u|_{H^{2}_{\text{mix}}}:=\left\|\frac{\partial^{2d}}{\partial
x_{1}^{2}\cdots\partial x_{d}^{2}}u\right\|^{2}_{2}$ is a
semi-norm~\cite{garcke}. In the classical theory of the
combination technique this estimate is used to prove the error bound
\begin{align}
\|u-u^{c}_{n}\|_{2} \leq \sum_{\|j\|_{1}>n}\|h_{j}\|_{2}&\leq 3^{-d}|u|_{H^{2}_{\text{mix}}}\sum_{k=n+1}^{\infty}2^{-2k}\binom{k+d-1}{d-1} \label{eqn:ccterr} \\
&=\frac{1}{3}\cdot 3^{-d} 2^{-2n} |u|_{H^{2}_{\text{mix}}}\sum_{k=0}^{d-1}\binom{n+d}{k}\left(\frac{1}{3}\right)^{d-1-k} \label{eqn:ccterr2} \\
&= \frac{1}{3}\cdot3^{-d}2^{-2n}|u|_{H^{2}_{\text{mix}}}\left(\frac{n^{d-1}}{(d-1)!}+\mathcal{O}(n^{d-2})\right) \label{eqn:ccterr3}
\end{align}
for sparse grid interpolation. 
Similar bounds can be shown for the $\infty$ and energy norms~\cite{griebel_bungartz}. 

In practice, strongly anisotropic grids, i.e.\ those with
$\|i\|_{\infty}\approx\|i\|_{1}$, can be problematic. Not only are they
difficult to compute in some circumstances but one also finds that they give
poor approximations that do not cancel out in the combination as expected.
Therefore it is often beneficial to implement a {\em truncated} combination. In
this paper we define a truncated combination as
\begin{equation}\label{eqn:tct}
u^{c}_{n,\tau}:=\sum_{k=0}^{d-1}(-1)^{k}\binom{d-1}{k}\sum_{\substack{\|i\|_{1}=n-k \\ \min(i)\geq \tau}}u_{i} \,,
\end{equation}
where $\tau$ is referred to as the truncation parameter. Such combinations will be used for the numerical results presented in Section~\ref{sec:numres}.

\subsection{A General Combination Technique}\label{sec:gct}

The combination technique can be generalised into arbitrary sums of solutions.
Given a (finite) set of multi-indices $I\subset\mathbb{N}^{d}$, one can write
\begin{equation}\label{eqn:gct}
u^{c}_{I}=\sum_{i\in I}c_{i}u_{i} \,,
\end{equation}
where the $c_{i}$ are referred to as the combination coefficients. 
It is easy to see that $u^{c}_{I}\in V^{s}_{I}:=\sum_{i\in I}V_{i}=\bigoplus_{i\in {I\downarrow}}W_{i}$ where ${I\downarrow}=\{i\in\mathbb{N}^{d}:\exists {j\in I}\text{ s.t. }i\leq j\}$. 
Not every choice of the $c_{i}$ will 
produce a reasonable approximation to $u$.
The question now is which combination coefficients produce the best
approximation to $u$? One could attempt to solve the optimisation problem of
minimising for example $\|u-\sum_{i\in I}c_{i}u_{i}\|_{2}$ which is known as {\em
Opticom}~\cite{hgc07}.
However this is cumbersome to solve in a massively parallel implementation and
also requires an approximation of the residual.
In this paper we consider combinations which we know {\em a priori} will give a
good approximation in some sense.
We define a {\em sensible} combination to be one in which each hierarchical
space contributes either once to the solution, or not at all. 
In particular we would like the coefficients to follow an inclusion{\slash}exclusion 
as described in Section~\ref{sec:cct}.
We observe that $W_{i}\subset V_{j}$ for all $j\geq i$, therefore one can easily 
determine how many times a given $W_{i}$ contributes to the solution by summing 
all of the coefficients $c_{j}$ for which $W_{i}\subset V_{j}$. 
In light of this we have the following definition.

\begin{definition}
A set of combination coefficients $\{c_{i}\}_{i\in I}$ is said to be {\bf\em
valid} if for each $i\in I$ it satisfies the property
$$
\sum_{j\in I,\, j\geq i}c_{j}\in\{0,1\} \,.
$$
We will also refer to a combination (of solutions) as valid if the corresponding
set of combination coefficients are valid.
\end{definition}

Another common assumption is that $\sum_{j\in I}c_{j}=1$. This relates to the definition if $\underline{0}\in I$ in which case we could assert that $1=\sum_{j\in I,\, j\geq \underline{0}}c_{j}$. 
The motivation for the definition is that this property is satisfied for dimension adaptive sparse grids~\cite{hegland_asg}. 
Let $P_{i}:V\mapsto V_{i}$ be a lattice of projection operators associated with the tensor product space $V$. 
As in $P_{i}$ satisfies $P_{i}P_{j}=P_{i\wedge j}$ (where $(i\wedge j)_{k}=\min\{i_{k},j_{k}\}$), $P_{i}P_{j}=P_{j}P_{i}$ and $P_{i}P_{i}=P_{i}$. 
Defining $P_{I}:V\mapsto V^{s}_{I}$ it follows from~\cite[p. C344]{hegland_asg} that for $i\in I$
\begin{equation*}
P_{i}P_{I}=P_{i}(1-\prod_{j\in I}(1-P_{j}))=P_{i}-P_{i}(1-P_{i})\prod_{j\in I\backslash\{i\}}(1-P_{j})=P_{i} \,.
\end{equation*} 
Since $P_{I}=\sum_{i\in I}c_{i}P_{i}$ it follows that $\sum_{\{j\in I\text{ s.t. }j\wedge i=i\}}c_{i}=\sum_{j\in I,\,j\geq i}c_{i}=1$.

For a given $I$ there are a many sets of valid combination coefficients that one
might take. To determine which of these {\em a priori} will give the best
approximation of $u$ we use approximations which are based on sparse grid
interpolation. It is reasonable to expect that such combinations will
work well in a more general setting given the underlying
inclusion{\slash}exclusion principle. 
The error estimate~\eqref{eqn:ccterr} for sparse grid interpolation is extended to the general combination technique by
$$
\|u-u^{c}_{I}\|_{2}\leq 3^{-d}|u|_{H^{2}_{\text{mix}}}\sum_{i\in\mathbb{N}^{d}}\left( 4^{-\|i\|_{1}}\left|1-\sum_{j\in I,\, j\geq i}c_{j}\right| \right) \,.
$$
Finding a set of valid coefficients which minimises this this bound is then equivalent to maximising
\begin{equation}\label{eqn:appqual}
Q(\{c_{i}\}_{i\in I}):=\sum_{i\in I\downarrow}4^{-\|i\|_{1}}\sum_{j\in I,\, j\geq i}c_{j} \,.
\end{equation}

Therefore, the general problem of finding the best combination of solutions
$u_{i}$ for $i\in I$ can be formulated as an optimisation problem, in particular
the maximisation of $Q(\{c_{i}\}_{i\in I})$ subject to the constraints
$\sum_{j\in I,\, j\geq i}c_{j}\in\{0,1\}$ for each $i\in I$. Since the $c_{i}$
must be integers this would be a simple integer linear programming (ILP)
problem if not for the non-trivial constraints. (To see why the $c_{i}$ must be
integers, we note that the non-zero $c_{i}$ for which $c_{j}=0$ for all $j>i$
must be $1$ in order for the set of coefficients to be valid. The remaining
coefficients are now obtained from the application of the
inclusion{\slash}exclusion principle which can only result in integer
coefficients.)

Fortunately we can simplify this by introducing the hierarchical coefficient.
\begin{definition}
Let $I$ be a set of multi-indices, then for $i\in I\downarrow$ we define the {\bf\em hierarchical coefficient}
\begin{equation}\label{eqn:wcm}
w_{i}:=\sum_{j\in I,\, j\geq i}c_{j} \,.
\end{equation}
\end{definition}

Suppose we expand our list of coefficients to the set $\{c_{i}\}_{i\in
I\downarrow}$ with the assumption $c_{i}=0$ for $i\notin I$. Now let $c,\,w$ be
vectors for the sets $\{c_{i}\}_{i\in I\downarrow},\,\{w_{i}\}_{i\in
I\downarrow}$, respectively (both having the same ordering with respect to $i\in
I\downarrow$). Using \eqref{eqn:wcm} we can write $w=Mc$ where $M$ is an
$|I\downarrow|\times|I\downarrow|$ matrix. Further, we note that if the elements
of $c,\,w$ are ordered according to ascending or descending values of
$\|i\|_{1}$ then $M$ is an upper or lower triangular matrix respectively
with $1$'s on the diagonal. Therefore $M$ is invertible and we have $c=M^{-1}w$.
Additionally, the restriction that $c_{i}=0$ for $i\notin I$ can be written as
$(M^{-1}w)_{i}=0$.

Since $w_{i}\in\{0,1\}$ for any set of valid coefficients we can formulate the {\em general
coefficient problem} (GCP) as the binary integer programming (BIP) problem of
maximising
\begin{equation}\label{eqn:aqw}
Q'(w):=\sum_{i\in I\downarrow}4^{-\|i\|_{1}}w_{i} 
\end{equation}
with the equality constraints $(M^{-1}w)_{i}=0$ for $i\notin I$. This is much
more manageable in practice and can be solved using a variety of algorithms that
are typically based on branch and bound, and/or cutting plane techniques.
However, this formulation also reveals that the general coefficient problem is
NP-complete~\cite{karp1972}. 
Equivalently, one can also think of this as a
weighted maximum satisfiability problem (Weighted MAX-SAT). If $I$ is a downset
then we can solve this rather quickly, but in general there exist cases which
take an incredibly long time to solve.
Another problem one runs into is that there is often not a unique solution. In
such circumstances we will simply pick any one of the solutions as they cannot
be further distinguished without additional information about $u$.

The only way to guarantee that a solution can be found quickly is to carefully
choose the index set $I$. One particular class of index sets of interest are
those which are closed under the $\wedge$ operator, that is if $i,j\in I$ then
$i\wedge j\in I$. In the theory of partially ordered sets, $(I,\leq)$ with this
property is referred to as a lower semi-lattice.
For such $I$ there is a unique solution to the GCP, namely $w_{i}=1$ for all
$i\in I\downarrow$ which clearly maximises~\eqref{eqn:aqw}. Computationally the
coefficients can be found quickly by first finding $\max I:=\{i\in I:\nexists
j\in I \text{ s.t. } j>i\}$, setting $c_{i}=1$ for $i\in\max I$, and then using
the inclusion{\slash}exclusion principle to find the remaining coefficients in
the order of descending $\|i\|_{1}$. This can also be viewed as an application
of the lattice theory of projections on function spaces presented by
Hegland~\cite{hegland_asg}.

Whilst the GCP presented is based upon $u\in H^{2}_{mix}$ we anticipate that the
resulting combinations will still yield reasonable results for larger function
spaces. This is based on the observation that the classical combination
technique has been successfully applied to a wide variety of problems for which
$u\notin H^{2}_{\text{mix}}$.

\begin{remark}
The restriction of $w_{i}$ to binary variables can be relaxed to
reals if we change the quantity we intend to optimise. The important observation
to make here is that one would expect having two contributions from a
hierarchical space is comparable to having no contributions. Further having a
fractional contribution like $\frac{1}{2}$ would be better than having no
contribution at all. In light of this we can try to solve the linear programming
problem of minimising
\begin{equation}\label{eqn:omwmin}
\sum_{i\in I\downarrow}4^{-\|i\|_{1}}|1-w_{i}|
\end{equation}
subject to the equality constraints $(M^{-1}w)_{i}=0$ for $i\notin I$. If $I$ is
closed under $\wedge$ then a minimum of $0$ is achieved for the same
hierarchical coefficients found in the binary formulation. In other cases the
solution to this relaxed optimisation problem is no worse than the solution to
the binary problem. This relaxation means the combination may no longer follow
an inclusion{\slash}exclusion principle. In practice the results are highly
dependant upon the true solution $u$ and the approximation properties of each of
the $u_{i}$. Additionally, the non-differentiability of equation \eqref{eqn:omwmin} means
that the problem is still non-trivial to solve in practice.


An approach that can sometimes speed up the computation of a solution to this
problem is to first find the solution to the quadratic programming problem of
minimising $$ \sum_{i\in I\downarrow}4^{-\|i\|_{1}}(1-w_{i})^{2} $$ subject to
the same equality constraints. This is a linear problem which is easily solved
using the method of Lagrange multipliers for example. In most circumstances we
would expect the solution of this problem to be close to the minimum of equation
\eqref{eqn:omwmin} and therefore make a good initial guess.
\end{remark}

\section{Fault Tolerant Combination Technique and Fault Simulation}\label{sec:faults} 


\subsection{Fault Tolerant Combination Technique (FTCT)}\label{sec:ftct}

In~\cite{my_ctac_paper,my_sga_paper} 
a fault tolerant combination technique was introduced.  
The most difficult aspect of generalising this work is the
updating of coefficients. Whilst some theory and a few simple cases have been
investigated, no general algorithm has been presented. 
Given the development of the general combination technique in Section~\ref{sec:gct} 
we are now able to consider a more complete theory of the FTCT.

Suppose we have a set $I$ of multi-indices for which we intend to
compute each of the solutions $u_{i}$ and combine as in \eqref{eqn:gct}. As each of
the $u_{i}$ can be computed independently the computation of these is easily 
distributed across different nodes of a high performance computer. 
Suppose that one or more of these nodes experiences a fault, hardware or
software in nature. As a result, some of our $u_{i}$ may not have been computed
correctly. We denote $J\subset I$ to be the set of indices for which the $u_{i}$
where not correctly computed. A lossless approach to fault tolerance would be to
recompute $u_{i}$ for $i\in J$. However, since recomputation is often costly,
we propose a lossy
approach to fault tolerance in which the failed solutions are not recomputed. In
this approach, rather than solving the generalised coefficient problem (GCP) for $I$,
we instead solve it for $I\backslash J$. As ${(I\backslash J)\downarrow}
\subseteq {I\downarrow}$ we expect this solution to have a larger error than
that if no faults had occurred. However, if $|J|$ is relatively small we would
also expect the loss of accuracy to be small because of the redundancy in the 
set of $\{u_{i}\}_{i\in I}$. 

As discussed in Section~\ref{sec:gct}, the GCP is difficult to solve in its
most general form. Whilst it can be solved rather quickly if the poset
$(I,\leq)$ is a lower semi-lattice, this is no longer any help in the FTCT since
the random nature of faults means we cannot guarantee that $(I\backslash
J,\leq)$ is always a lower semi-lattice. The only way we could ensure this is 
to restrict which elements of $I$ can be in $J$. A simple way
to achieve this is to recompute missing $u_{i}$ 
if $(I\backslash\{i\},\leq)$ is not a lower semi-lattice. In particular this
is achieved if all $u_{i}$ with $i\notin\max I$ are recomputed. Since elements in $\max I$
correspond to the solutions on the largest of the grids, we are 
avoiding the recomputation of the solutions which take the longest to
compute. 
Additionally, this also means only the largest of the
hierarchical spaces are ever omitted as a result of a failure. As these
contribute the least to the solution we expect the resulting error to be
relatively close to that of the solution if no faults had occurred. Finally,
since $(I\backslash J,\leq)$ is then a lower semi-lattice, the resulting GCP for
$I\backslash J$ has a unique maximal solution which is easily computed.

We now illustrate this approach as it is applied to the classical
combination technique. We define $I_{n}=\{i\in\mathbb{N}^{d}:\|i\|_{1}\leq n\}$.
It was shown in~\cite{my_ctac_paper} that the proportion of additional unknowns 
in computing the
solutions $u_{i}$ for all $i\in I_{n}$ compared to $n-d<\|i\|_{1}\leq n$ is at
most $\frac{1}{2^{d}-1}$. 
If no faults occur then the combination is exactly the
classical combination technique with $c_{i}=(-1)^{n-\|i\|_{1}}\binom{d-1}{
n-\|i\|_{1}}$ if $n-d<\|i\|\leq n$ and $c_{i}=0$ otherwise. If faults do occur
then we recompute any $u_{i}$ with $\|i\|_{1}<n$ that was not successfully
computed. If no faults occurred for any $u_{i}$ with $\|i\|=n$ then we can again
proceed with the classical combination. If faults affect any $u_{i}$ with
$\|i\|_{1}=n$ then we add such $i$ to the set $J$ and then solve the GCP for
$I_{n}\backslash J$. The solution is trivially obtained with hierarchical coefficients 
$w_{i}=1$ for all $i\in I_{n}\backslash J$.

The largest solutions (in terms of unknowns) which may have to be recomputed 
are those with 
$\|i\|_{1}=n-1$ which would be expected to take at most half the time of those
solutions with $\|i\|_{1}=n$. Since they take less time to compute they are also
less likely to be lost due to failure. Additionally, there are $\binom{n-1+d-1}{d-1}$ 
solutions with $\|i\|_{1}=n-1$ which is less than the $\binom{n+d-1}{d-1}$
with $\|i\|_{1}=n$. As a result of these observations, we would expect to see
far less disruptions caused by recomputations when using this approach compared
to a lossless approach where all failed solutions are recomputed.

The worst case scenario with this approach is that all $u_{i}$ with $\|i\|_{1}=n$ 
are not successfully computed due to faults. In this case the resulting combination is 
simply a classical combination of level $n-1$. This only requires the solutions
$u_{i}$ with $n-d\leq \|i\|\leq n-1$. Likewise, all solutions to the GCP in this
approach result in zero coefficients for all $c_{i}$ with $i< n-d$. We can
therefore reduce the overhead of the FTCT by only computing the solutions
$u_{i}$ for $n-d\leq\|i\|_{1}\leq n$. It is known that the proportion of additional 
unknowns compared to the classical combination technique 
in this case is at most $\frac{1}{2(2^{d}-1)}$~\cite{my_ctac_paper}.

The solutions $u_{i}$ with $\|i\|_{1}=n-1$ are only half the size of the largest $u_{i}$ and hence recomputation of these may also be disruptive and undesirable. 
We could therefore consider recomputing only solutions with $\|i\|_{1}\leq n-2$. 
By doing this the recomputations are even more manageable having at most 
one quarter the unknowns of the largest $u_{i}$.
The worst case here is that all solutions with $i\geq n-1$ fail and we
end up with a classical combination of level $n-2$. Again it turns out 
one does not require the entire downset $I_{n}$, in this case the (modified) FTCT 
requires solutions $u_{i}$ with $n-d-1\leq\|i\|_{1}\leq n$. 
Using arguments similar to those
in~\cite{my_ctac_paper} it is easily shown that the overhead in this case is at most 
$\frac{3}{4(2^{d}-1)}$. The trade-off now is that the update of coefficients
takes a little more work. We are back in the situation where we cannot guarantee
that $(I_{n}\backslash J,\leq)$ is a lower semi-lattice.

To solve the GCP in this case we start with all $w_{i}$ equal to $1$. If
failures affected any $u_{i}$ with $\|i\|_{1}=n$ we set the corresponding
constraints $c_{i}=w_{i}=0$.
For failures occurring on $u_{i}$ with $\|i\|_{1}=n-1$ we have the constraints
$w_{i}-\sum_{k=1}^{d}w_{i+e^{k}}=0$ (with $e^{k}$ being the multi-index with
$e^{k}_{l}=\delta_{k,l}$). We note that (since the $w_{i}$ are binary variables)
this can only be satisfied if at most one of the $w_{i+e^{j}}$ is equal to $1$.
Further, if $\sum_{k=1}^{d}w_{i+e^{k}}=0$ we must also have $w_{i}=0$. This
gives us a total of $d+1$ feasible solutions to check for each such constraint. 
Given $g$ failures on solutions with $\|i\|_{1}=n-1$ we have at most 
$(d+1)^{g}$ feasible solutions to the GCP to check. 
This can be kept manageable if solutions are combined frequently enough that 
the number of failures $g$ that are likely occur in between is small.
One solves the GCP by computing the objective function~\eqref{eqn:aqw} for each 
of the feasible solutions identified and selecting one which maximises this. 
Where some of the failures on the second layer are sufficiently far apart on the
lattice, it is possible to significantly reduce the number of cases to check as 
constraints can be optimised independently. 

We could continue and describe an algorithm for only recomputing the fourth
layer and below, however the coefficient updates here begin to become much more
complex (both to describe and to compute). Our experience indicates that the
recomputation of the third layer and below is a good trade-off between the need
to recompute and the complexity of updating the coefficients. Our numerical results in Section~\ref{sec:numres} are obtained using this approach.

\subsection{Probability of failure for computations}\label{sec:pfc}

To analyse the expected outcome of the fault tolerant combination technique described in 
Section~\ref{sec:ftct} we need to know the probability of each $u_{i}$ failing. 
In particular, the availability of $u_{i}$ will be modelled as a simple Bernoulli 
process $U_{i}$ which is $0$ if $u_{i}$ was computed successfully and is $1$ otherwise. 
It is assumed that each $u_{i}$ is computed on a single computational node. 
Therefore we are interested in the probability that a failure occurs on this node before 
the computation of $u_{i}$ is complete, that is $\Pr(U_{i}=1)$. 
Suppose $T$ is a random variable denoting the time to failure on a given node and the time 
required to compute $u_{i}$ is given by $t_{i}$, then one has 
$\Pr(U_{i}=1)=\Pr(T\leq t_{i})$. 
One therefore needs to know something about the distribution of $T$.

Schroeder and Gibson analysed the occurrence of faults on $22$ high
performance machines at LANL from 1996-2005~\cite{schroeder_gibson}. They found
that the distribution of time between failures for a typical node in the
systems studied was best fit by the Weibull distribution with a shape parameter of
$0.7$. 
Based upon this study we will consider a model of faults on each node based upon the Weibull 
renewal process, that is a renewal process where inter-arrival times are Weibull distributed, 
with shape parameter $0<\kappa\leq 1$.

There are several reasons for considering a renewal process for modelling faults. 
First, renewal theory is commonly used in availability analysis and there are many 
extensions such as alternating renewal processes in which one can also consider repair times. 
Second, we expect a fault tolerant implementation of {\sc mpi} to enable the substitution 
of a failed node with another available node in which case computation can continue from 
some recovered state. This will be further discussed in Section~\ref{sec:fs}.
We now derive the value of $\Pr(U_{i}=1)$. 

Let $\{X_{k}\}_{k=1}^{\infty}$ be random variables for the successive times between failures on a node. 
We assume that the $X_{k}$ are positive, independent and identically distributed with cumulative distribution
\begin{equation}\label{eqn:wrv}
F(t):=\Pr(X_{k}\leq t)
=1-e^{-(t/\lambda)^{\kappa}} 
\end{equation}
for some $0<\lambda<\infty$ and $0<\kappa\leq1$. 
Let $S_{k}=\sum_{m=1}^{k}X_{m}$ (for $k\geq1$) be the waiting time to the $k$th failure. 
Let $N(t)$ count the number of failures that have occured up until (and including) time $t$, 
that is $N(t)=\max\{k:S_{k}\leq t\}$). By the elementary renewal theorem one has 
\begin{equation}\label{eqn:ert}
\lim_{t\rightarrow\infty}\frac{1}{t}\mathbb{E}[N(t)]=\frac{1}{\mathbb{E}[X_{1}]}=\frac{1}{\lambda\Gamma(1+\frac{1}{\kappa})} \,.
\end{equation}
We thus get an expression for the (long term) average rate of faults. 

Now, whilst we have a distribution for the time between failures, when a 
computation starts it is generally unknown how much time has elapsed since the 
last failure occurred. Hence our random variable $T$ is what is referred to as the 
random incidence (or residual lifetime). Noting that one is more likely to intercept 
longer intervals of the renewal process than shorter ones and that the probability 
distribution of the starting time is uniform over the interval, 
then it is straightforward to show~\cite{trivedi} that $T$ has density 
$$
g(t)=\frac{1-F(t)}{\mathbb{E}[X_{1}]}=\frac{e^{-(t/\lambda)^{\kappa}}}{\lambda\Gamma(1+\frac{1}{\kappa})} \,.
$$
It follows that the cumulative probability distribution 
$$
G(t_{i}):=\Pr(T\leq t_{i})
=\frac{1}{\lambda\Gamma(1+\frac{1}{\kappa})}\int_{0}^{t_{i}}e^{-(x/\lambda)^{\kappa}}dx \,.
$$
The resulting distribution has similar properties to the original Weibull distribution. 
In fact, when $\kappa=1$ we note that $X_{k}$ and $T$ are identically distributed, 
they are exponential with mean $\lambda$. 
Further, for $0<\kappa\leq 1$ we have the following bound:
\begin{lemma}\label{lem:rib}
For $0<\kappa\leq 1$ one has
\begin{equation}\label{eqn:ribbw}
G(t)\leq F(t) \,.
\end{equation}
\end{lemma}
\begin{proof}
We note that via a change of variables that
$$
\Gamma\left(1+\frac{1}{\kappa}\right)=\int_{0}^{\infty}y^{\frac{1}{\kappa}}e^{-y}dy=\int_{0}^{\infty}\frac{\kappa}{\lambda}\left(\frac{x}{\lambda}\right)^{\kappa}e^{-(x/\lambda)^{\kappa}} dx
$$
and therefore
\begin{align*}
\Gamma\left(1+\frac{1}{\kappa}\right)e^{-(t/\lambda)^{\kappa}}
&=\int_{0}^{\infty}\frac{\kappa}{\lambda}\left(\frac{x}{\lambda}\right)^{\kappa}e^{-(x/\lambda)^{\kappa}}e^{-(t/\lambda)^{\kappa}} dx \\
&= \int_{0}^{\infty}\frac{x}{\lambda}\frac{\kappa}{\lambda}\left(\frac{x}{\lambda}\right)^{\kappa-1}e^{-(x/\lambda)^{\kappa}}e^{-(t/\lambda)^{\kappa}} dx \\
&= \left[-\frac{x}{\lambda}e^{-(x/\lambda)^{\kappa}}e^{-(t/\lambda)^{\kappa}}\right]_{0}^{\infty} -\int_{0}^{\infty}-\frac{1}{\lambda}e^{-(x/\lambda)^{\kappa}-(t/\lambda)^{\kappa}} dx \\
&=\int_{0}^{\infty}\frac{1}{\lambda}e^{-(x/\lambda)^{\kappa}-(t/\lambda)^{\kappa}} dx \,.
\end{align*}
Since $0<\kappa\leq 1$ and $t,x\geq 0$ one has $x^{\kappa}+t^{\kappa}\geq(x+t)^{\kappa}$ and hence
\begin{align*}
\Gamma\left(1+\frac{1}{\kappa}\right)e^{-(t/\lambda)^{\kappa}}
&\leq\int_{0}^{\infty}\frac{1}{\lambda}e^{-((x+t)/\lambda)^{\kappa}} dx \\
&=\int_{t}^{\infty}\frac{1}{\lambda}e^{-(x/\lambda)^{\kappa}} dx \\
&=\int_{0}^{\infty}\frac{1}{\lambda}e^{-(x/\lambda)^{\kappa}} dx -\int_{0}^{t}\frac{1}{\lambda}e^{-(x/\lambda)^{\kappa}} dx\\
&=\Gamma\left(1+\frac{1}{\kappa}\right)-\frac{1}{\lambda}\int_{0}^{t}e^{-(x/\lambda)^{\kappa}} dx \,.
\end{align*}
Rearranging gives
$$
\frac{1}{\lambda\Gamma(1+\frac{1}{\kappa})}\int_{0}^{t}e^{-(x/\lambda)^{\kappa}}dx \leq 1-e^{-(t/\lambda)^{\kappa}} 
$$
which is the desired inequality.
\qquad\end{proof}

As a result of Lemma~\ref{lem:rib} one has that the probability of $u_{i}$ failing to compute successfully is bounded above by
$$
\Pr(U_{i}=1)=G(t_{i})\leq F(t_{i}) \,.
$$

\begin{remark}
The result of Lemma~\ref{lem:rib} is essentially a consequence of the 
property 
\begin{equation}\label{eqn:wpb}
\Pr(X\leq s+t \mid X>s)\leq \Pr(X\leq t) 
\end{equation}
where $X$ is Weibull distributed with shape parameter $0<\kappa\leq 1$ and $s,t\geq 0$.
This property can be extended to the fact that if $s_{2}\geq s_{1}\geq0$ then 
$$ 
\Pr(X\leq s_{2}+t \mid
X\geq s_{2})\leq\Pr(X\leq s_{1}+t \mid X\geq s_{1}) \,.
$$ 
This has important implications on the order in which we compute successive solutions on a single node. 
Solutions one is least concerned about not completing due to a fault should 
be computed first and solutions for which we would like to minimise the chance of failure should be computed later. 
\end{remark}
\begin{remark}
We note that for $\kappa\geq 1$ the inequalities of Equations~\eqref{eqn:ribbw} 
and~\eqref{eqn:wpb} are reversed thus $G(t)\geq F(t)$ and 
$\Pr(X\leq s+t \mid X>s)\geq \Pr(X\leq t)$ for $s,t\geq 0$.
\end{remark}

\subsection{Fault Simulation in the FTCT algorithm}\label{sec:fs}

We first describe the parallel FTCT algorithm:
\begin{itemize}
\item[1.] 
Given a (finite) set of multi-indices $I$, distribute the computation of 
component solutions $u_{i}$ for $i\in I$ amongst the available nodes. 
\item[2.] 
Each node begins to compute the $u_{i}$ which have been assigned to it. 
For time evolving {\sc pde}'s the solvers are evolved for some fixed simulation time $t_{s}$.
\item[3.] 
On each node, once a $u_{i}$ is computed a checkpoint of the result is saved. 
If the node experiences a fault during the computation of a $u_{i}$, a fault tolerant implementation of {\sc mpi} is used to replace this node with another 
(or continue once the interrupted node is rebooted). 
On the new node, checkpoints of previously computed $u_{i}$ are loaded and we then assess 
whether the interrupted computation should be recomputed or discarded. If it is to be 
recomputed this is done before computing any of the remaining $u_{i}$ allocated to the node.
\item[4.] 
Once all nodes have completed their computations they communicate which $u_{i}$ 
have been successfully computed via a {\sc mpi\_allreduce}. All nodes now have a list of 
multi-indices $I'\subseteq I$ and can solve the {\sc gcp} to obtain combination coefficients.
\item[5.] 
All nodes compute a partial sum $c_{i}u_{i}$ for the $u_{i}$ that it has computed
and then the sum is completed globally via a {\sc mpi\_allreduce} such that 
all nodes now have a copy of $u^{\text{gcp}}_{I'}$.
\item[6.] 
In the case of a time evolving {\sc pde} the $u_{i}$ can be sampled from 
$u^{\text{gcp}}_{I'}$ and the computation further evolved by repeating from 2.
\end{itemize}

The optional step 6 for time evolution problems 
has many advantages. First, by combining component 
solutions several times throughout the computation one can improve the approximation 
to the true solution. Second, each combination can act like a global checkpoint such that 
when a $u_{i}$ fails it can be easily restarted from the last combination 
(rather than the very beginning). Third, there are potential opportunities to reassess 
the load balancing after each combination and potentially re-distribute the $u_{i}$ 
to improve performance in the next iteration.

For the numerical results in Section~\ref{sec:numres} we do not currently use a fault tolerant implementation of {\sc mpi} and instead simulate faults by modifying the following steps:
\begin{itemize}
\item[2.] 
Before each node begins computing the $u_{i}$ assigned to it, a process on the node computes a (simulated) time of failure by sampling a distribution for time to failure and adding it to the current time. In our results we sample the Weibull distribution for some mean $\lambda>0$ and shape $0<\kappa\leq 1$. 
\item[3.]
Immediately after a $u_{i}$ has been computed on a node we check to see if the current time has exceeded the (simulated) time of failure. If this is the case the most recent computation is discarded. We then pretend that the faulty node has been instantly replaced and continue with step 3 as described.
\end{itemize}


Note from Section~\ref{sec:pfc} that sampling the Weibull distribution produces faults at least as often as the random incidence $G(t)$. Thus, by sampling the Weibull distribution in our simulation, the results should be no worse than what might occur in reality allowing for some small discrepancy in fitting the Weibull distribution to existing data of time between failures on nodes of a real machine. 

The assumption in step 3 of replacing a failed node with another is based upon what 
one might expect from a fault tolerant {\sc mpi}. In fact both Harness FT-MPI\footnote{http://icl.cs.utk.edu/ftmpi/} 
and the relatively new ULFM specification\footnote{http://fault-tolerance.org/} allow this, 
although it certainly does not occur in an instant as is assumed in our simulation. 
Due to limited data at the current time we are unable to predict what recovery times 
one might expect. 
We also note that (simulated) failures are checked for at the completion of the computation of 
each $u_{i}$. Since a failure is most likely to occur some time before the computation 
completes then time is wasted in the simulation from the sampled time of failure to 
the completion of the affected computation. 
Improving these aspects of the simulation and implementation with a fault tolerant {\sc mpi} 
will be the subject of future work.

\subsection{Expected error of the FTCT}\label{sec:iea}

In this section, we bound the expected interpolation error for the {\sc ftct} as
applied to the classical combination technique as described in
Section~\ref{sec:ftct}. In particular we look at the case where all solutions
with $\|i\|_{1}<n$ are recomputed, and the case where all solutions with
$\|i\|_{1}<n-1$ are recomputed as described in Section~\ref{sec:ftct}. 

Given $u\in H^{2}_{\text{mix}}$, let
$$
\epsilon_{n}:=\frac{1}{3}\cdot 3^{-d} 2^{-2n} |u|_{H^{2}_{\text{mix}}}\sum_{k=0}^{d-1}\binom{n+d}{k}\left(\frac{1}{3}\right)^{d-1-k} 
$$
such that  $\|u-u^{c}_{n}\|_{2}\leq \epsilon_{n}$, see~\eqref{eqn:ccterr}. 
Now given a (finite) set of multi-indices $I$ we denote $u^{\text{gcp}}_{I}$ to be a
combination $\sum_{i\in I}c_{i}u_{i}$ 
which is a solution to the {\sc gcp} described in Section~\ref{sec:gct}. 
When faults prevent successful computation of some of the $u_{i}$ we must 
find $u^{\text{gcp}}_{I'}$ for some $I'\subset I$.
Consider the Bernoulli process $\{U_{i}\}_{i\in I}$ for 
which each $U_{i}=0$ if $u_{i}$ is computed successfully and is $1$ otherwise 
as described in Section~\ref{sec:pfc}. 
Additionally it is assumed that the computation of each $u_{i}$ is done within one node, 
that is many $u_{i}$ can be computed simultaneously on different nodes but each individual 
$u_{i}$ is computed within one hardware node. 
Let $t_{i}$ be the time required to compute $u_{i}$ for each $i\in I$. 
We assume that the time between failures on each node is Weibull distributed. 
As demonstrated in Section~\ref{sec:pfc} the probability of each $u_{i}$ being lost 
a the result of a fault is given by the random incidence distribution
$$
\Pr(U_{i}=1)=G(t_{i})\leq F(t_{i}) \,.
$$ 
Given that $u_{i}$ with the same $\|i\|_{1}$ have a similar number of unknowns 
we assume 
they will take roughly the same amount of time to compute. 
We therefore define $t_{k}:=\max_{\|i\|_{1}=k}t_{i}$, that is the maximal
time to compute any $u_{i}$ with level $k$.

As a result, for each $i\in I$ the probability of each $u_{i}$ 
not completing due to failures is bounded by 
$$
\Pr(U_{i}=1)\leq G(t_{\|i\|_{1}})\leq F(t_{\|i\|_{1}}) \,.
$$.
With this we can now give the main result. 
\begin{proposition}\label{prop:err1}
Given $d,n>0$ and $I_{n}:=\{ i\in\mathbb{N}^{d}:\|i\|_{1}\leq n\}$ 
let $u_{i}$ be the interpolant of $u\in H^{2}_{\text{mix}}$ for $i\in I_{n}$. 
Let each $u_{i}$ be computed on a different node of a parallel computer 
for which the time between failures on every node is independent and identically 
Weibull distributed with mean $\lambda>0$ and shape parameter $0<\kappa\leq 1$.
Let $t_{i}$ be the (wall) time required to compute each $u_{i}$ 
and $t_{n}=\max_{\|i\|_{1}=n}t_{i}$. 
Suppose we recompute any $u_{i}$ with $\|i\|_{1}<n$ which is interrupted by a fault, 
let $\mathcal{I}$ be the set of all possible $I'\subseteq I_{n}$ for which $u_{i}$ was 
successfully computed (eventually) iff $i\in I'$. 
Let $u^{\text{gcp}}_{\mathcal{I}}$ be the function-valued random variable corresponding to 
the result of the {\sc ftct} (i.e. $u^{\text{gcp}}_{I'}$ for some random $I'\in\mathcal{I}$),
then the expected error is bounded above by
\begin{equation*}
\mathbb{E}\left[\|u-u^{\text{gcp}}_{\mathcal{I}}\|_{2}\right]
\leq\epsilon_{n}\left(1+3\left(1-e^{-(t_{n}/\lambda)^{\kappa}}\right)\right) \,.
\end{equation*}
\end{proposition}
\begin{proof}
Since $u_{i}$ with $\|i\|_{1}<n$ are recomputed we have that $U_{i}=0$ 
for all $\|i\|_{1}<n$ and therefore $\Pr(\mathcal{I}=I')=0$ for $I_{n-1}\nsubseteq I'$. 
Note that $I_{k}$ is a downset for $k\geq 0$, that is $I_{k}=I_{k}\downarrow$. 
Since the $i$ with $\|i\|_{1}=n$ are covering elements for $I_{n-1}$ it follows 
that each of the $I'$ for which $\Pr(\mathcal{I}=I')>0$ are also downsets. 
It follows that there is a unique solution to the GCP for such $I'$, 
namely $w_{i}=1$ for all $i\in I'$ and $w_{i}=0$ otherwise. 
That is $w_{i}=0$ iff $U_{i}=1$ and hence $w_{i}=1-U_{i}$. 
It follows that the error is bounded by 
$$
\|u-u^{\text{gcp}}_{I'}\|_{2}
\leq\|u-u^{c}_{n}\|_{2}+\sum_{\|i\|_{1}=n}\left|1-w_{i}\right|\|h_{i}\|_{2}=\|u-u^{c}_{n}\|_{2}+\sum_{\|i\|_{1}=n}U_{i}\|h_{i}\|_{2}\,.
$$
From Lemma~\ref{lem:rib}, the probability of a fault occurring 
during the computation of any $u_{i}$ with $\|i\|_{1}=n$ is bounded by 
$G(t_{n})$ and therefore for $\|i\|_{1}=n$ one has
\begin{align*}
\mathbb{E}[U_{i}]=0\cdot\Pr(U_{i}=0)+1\cdot\Pr(U_{i}=1)=\Pr(U_{i}=1)=G(t_{i})\leq G(t_{n}) \,.
\end{align*}
It follows that
\begin{align}
\mathbb{E}\left[\|u-u^{\text{gcp}}_{\mathcal{I}}\|_{2}\right]
&\leq\mathbb{E}\left[\|u-u^{c}_{n}\|_{2}+\sum_{\|i\|_{1}=n}U_{i}\|h_{i}\|_{2}\right] \nonumber \\
&=\|u-u^{c}_{n}\|_{2}+\sum_{\|i\|_{1}=n}\mathbb{E}[U_{i}]\|h_{i}\|_{2} \nonumber \\
&\leq\|u-u^{c}_{n}\|_{2}+\sum_{\|i\|_{1}=n}G(t_{i})\|h_{i}\|_{2}  \,, \label{eqn:pro1est}
\end{align}
and substituting the estimate $\|h_{i}\|_{2}\leq3^{-d}2^{-2\|i\|_{1}}|u|_{H^{2}_{\text{mix}}}$ yields
\begin{align*}
\mathbb{E}\left[\|u-u^{\text{gcp}}_{\mathcal{I}}\|_{2}\right]
&\leq\epsilon_{n}+\sum_{\|i\|_{1}=n}G(t_{n}) 3^{-d}2^{-2n}|u|_{H^{2}_{\text{mix}}} \\
&\leq\epsilon_{n}+F(t_{n})\sum_{\|i\|_{1}=n} 3^{-d}2^{-2n}|u|_{H^{2}_{\text{mix}}} \\
&=\epsilon_{n}+\left(1-e^{-(t_{n}/\lambda)^{\kappa}}\right)\binom{n+d-1}{d-1} 3^{-d}2^{-2n}|u|_{H^{2}_{\text{mix}}} \,.
\end{align*}
Now, noting that $\binom{n+d-1}{d-1}\leq\sum_{k=0}^{d-1}\binom{n+d}{k}(1/3)^{d-1-k}$ one has 
\begin{equation*}
\binom{n+d-1}{d-1}3^{-d}2^{-2n}|u|_{H^{2}_{\text{mix}}}\leq 3\epsilon_{n}
\end{equation*}
and therefore, 
\begin{equation}
\mathbb{E}\left[\|u-u^{\text{gcp}}_{\mathcal{I}}\|_{2}\right]\leq\epsilon_{n}\left(1+3\left(1-e^{-(t_{n}/\lambda)^{\kappa}}\right)\right) \,.
\label{eq:eeb1}\end{equation}
Note that as $t_{n}/\lambda\rightarrow\infty$ we have 
$\mathbb{E}\left[\|u-u^{\text{gcp}}_{\mathcal{I}}\|_{2}\right]\leq4\epsilon_{n}$.
However, the worst case scenario is when $I'=I_{n-1}$ which results 
in a classical combination of level $n-1$ which has the error bound
\begin{align*}
\|u-u^{c}_{n-1}\|_{2}\leq\epsilon_{n-1}&=\frac{1}{3}\cdot 3^{-d} 2^{-2(n-1)} |u|_{H^{2}_{\text{mix}}}\sum_{k=0}^{d-1}\binom{n-1+d}{k}\left(\frac{1}{3}\right)^{d-1-k} \\
&\leq\frac{4}{3}\cdot 3^{-d} 2^{-2n} |u|_{H^{2}_{\text{mix}}}\sum_{k=0}^{d-1}\binom{n+d}{k}\left(\frac{1}{3}\right)^{d-1-k}  \\
&=4\cdot\epsilon_{n} \,.
\end{align*}
This is consistent with the upper bound~\eqref{eq:eeb1} which is the desired result.
\qquad \end{proof}

Note the assumption that each $u_{i}$ be computed on a different node is not necessary as we have bounded the probability of a failure during the computation of $u_{i}$ to be independent of the starting time. As a result our bound is independent of the number of nodes that are used during the computation and how the $u_{i}$ are distributed among them as long as each individual $u_{i}$ is not distributed across multiple nodes. 

The nice thing about this result is that the bound on the expected
error is simply a multiple of the error bound for $u^{c}_{n}$, i.e. the result
in the absence of faults. If the bound on $\|u-u^{c}_{n}\|_{2}$ was tight then one might expect
$$
\mathbb{E}\left[\|u-u^{\text{gcp}}_{\mathcal{I}}\|_{2}\right]\lessapprox \|u-u^{c}_{n}\|_{2}\left(1+3\left(1-e^{-(t_{n}/\lambda)^{\kappa}}\right)\right) \,.
$$
Also note that~\eqref{eqn:pro1est} can be expressed as
\begin{equation*}
\mathbb{E}\left[\|u-u^{\text{gcp}}_{\mathcal{I}}\|_{2}\right]
\leq\|u-u^{c}_{n}\|_{2}+\Pr(T\leq t_{n})\|u^{c}_{n-1}-u^{c}_{n}\|_{2} \,.
\end{equation*}
If the combination technique converges for $u$ then $\|u^{c}_{n-1}-u^{c}_{n}\|_{2}\rightarrow0$ as $n\rightarrow\infty$. Since $\Pr(T\leq t_{n})\leq1$ the error due to faults diminishes as $n\rightarrow\infty$. 
We now prove an analogous result for the case where only solutions with $\|i\|_{1}<n-1$ are recomputed.

\begin{proposition}
Given $d,n>0$ and $I_{n}:=\{ i\in\mathbb{N}^{d}:\|i\|_{1}\leq n\}$ 
let $u_{i}$, $t_{i}$ and $t_{n}$ be as described in Proposition~\ref{prop:err1} with each 
$u_{i}$ computed on different nodes for which time between failures 
is iid having Weibull distribution with $\lambda>0$ 
and $0<\kappa\leq 1$. 
Additionally let $t_{n-1}=\max_{\|i\|_{1}=n-1}t_{i}$. 
Suppose we recompute any $u_{i}$ with $\|i\|_{1}<n-1$ which is interrupted by a fault, 
let $\mathcal{I}$ be the set of all possible $I'\subseteq I_{n}$ for which $u_{i}$ was 
successfully computed (eventually) iff $i\in I'$. 
Let $u^{\text{gcp}}_{\mathcal{I}}$ be the function-valued random variable corresponding to 
the result of the {\sc ftct},
then 
$$
\mathbb{E}\left[\|u-u^{\text{gcp}}_{\mathcal{I}}\|_{2}\right]\leq\epsilon_{n}\cdot\min\left\{16,1+3\left(d+5-e^{-\left(\frac{t_{n}}{\lambda}\right)^{\kappa}}-(d+4) e^{-\left(\frac{t_{n-1}}{\lambda}\right)^{\kappa}}\right)\right\} \,.
$$
\end{proposition}
\begin{proof}
This is much the same as the proof of Proposition~\ref{prop:err1}. The solution to the {\sc gcp} for $I_{n}$ satisfies the property that if $w_{i}=0$ for $\|i\|_{1}=n-1$ then $u_{i}$ was not computed successfully, that is $U_{i}=1$. However the converse does not hold in general. Regardless, if $U_{i}=1$ for $\|i\|_{1}=n-1$ then the worst case is that $w_{i}=0$ and $w_{j}=0$ for the $d$ possible $\|j\|_{1}=n$ satisfying $j>i$. We therefore note that
the error generated by faults affecting $u_{i}$ with $\|i\|_{1}=n-1$ is bounded
by
\begin{equation}\label{eqn:l2b}
\sum_{j\in I_{n},\,j\geq i}\|h_{j}\|_{2}\leq (d+4)3^{-d}2^{-2n}|u|_{H^{2}_{\text{mix}}} \,.
\end{equation}
Therefore we have 
\begin{align*}
\mathbb{E}\left[\|u-u^{\text{gcp}}_{\mathcal{I}}\|_{2}\right]
&\leq\|u-u^{c}_{n}\|_{2}+\sum_{\|i\|_{1}=n}G(t_{n})\|h_{i}\|_{2} \\
&\qquad+\sum_{\|i\|_{1}=n-1}G(t_{n-1})\sum_{j\in I,\,j\geq i}\|h_{j}\|_{2} \\
&\leq\epsilon_{n}\left(1+3\left(1-e^{-(t_{n}/\lambda)^{\kappa}}\right)\right) \\
&\qquad+\binom{n-1+d-1}{d-1}\left(1-e^{-(t_{n-1}/\lambda)^{\kappa}}\right)(d+4)3^{-d}2^{-2n}|u|_{H^{2}_{\text{mix}}} \\
&\leq\epsilon_{n}\left(1+3\left(1-e^{-(t_{n}/\lambda)^{\kappa}}\right) +3(d+4)\left(1-e^{-(t_{n-1}/\lambda)^{\kappa}}\right)\right) \,.
\end{align*}
Now the expected error should be no more than the worse case which is
$I'=I_{n-2}$ for which we have $\|u-u^{c}_{n-2}\|_{2}\leq
16\epsilon_{n}$. Taking the minimum of the two estimates yields the desired
result.
\qquad\end{proof}

To illustrate how this result may be used in practice, suppose we compute a level
$12$ interpolation in $3$ dimensions on a machine whose mean time to failure 
can be modelled by the Weibull distribution with a mean of $100$ seconds and 
shape parameter $0.7$. Further, suppose $u_{i}$ with $\|i\|_{1}>10$ are not 
recomputed if lost as a result of a fault and that $t_{12}$ is estimated to be $1.0$ 
seconds and $t_{11}$ is at most $0.5$ seconds. The expected error for our 
computation is bounded above by $1.63$ times the error bound if no
faults were to occur.

Whilst this provides some theoretical validation of our approach, in practice we can 
numerically compute an improved estimate by enumerating all possible outcomes and the probability of each occurring. 
The reason for this is that Equation~\eqref{eqn:l2b} is an overestimate in general, 
particularly for relatively small $d$. In practice, a fault on $u_{i}$ with $\|i\|_{1}=n-1$
will generally result in the loss of $d-1$ of the largest hierarchical spaces in
which case Equation~\eqref{eqn:l2b} overestimates by a factor of $\frac{d+4}{d-1}$.



\subsection{Expected computation time}

We now repeat the above analysis, this time focusing on the mean time required
for recomputations. The first issue to consider is that a failure may occur
during a recomputation which will trigger another recomputation. Given a
solution $u_{i}$ with $\|i\|_{1}=m$, the probability of having to recompute $r$
times is bounded by $G(t_{m})^{r}\leq F(t_{m})^{r}$. Hence the expected number of
recomputations for such a $u_{i}$ is bounded by
$$
\sum_{r=1}^{\infty}rF(t_{m})^{r}=\frac{F(t_{m})}{(1-F(t_{m}))^{2}} =e^{(t_{m}/\lambda)^{\kappa}}(e^{(t_{m}/\lambda)^{\kappa}}-1) \,. 
$$
Let the time required to compute each $u_{i}$ be bounded by $t_{i}\leq c2^{\|i\|_{1}}$ for some fixed $c>0$ and all $\|i\|_{1}\leq n$. 
For a given $m\leq n$, suppose we intend to recompute all $u_{i}$ with $\|i\|_{1}\leq m$ upon failure,
then the expected time required for recomputations is bounded by
$$
R_{m}
\leq\sum_{\|i\|_{1}\leq m}t_{\|i\|_{1}}\sum_{r=1}^{\infty}rF(t_{\|i\|_{1}})^{r}
\leq\sum_{k=0}^{m}\binom{k+d-1}{d-1}c2^{k} e^{(c2^{k}/\lambda)^{\kappa}}\left(e^{(c2^{k}/\lambda)^{\kappa}}-1\right) \,.
$$
and by bounding components of the sum with the case $k=m$ one obtains
\begin{align*}
R_{m}&\leq\binom{m+d-1}{d-1} e^{(c2^{m}/\lambda)^{\kappa}}\left(e^{(c2^{m}/\lambda)^{\kappa}}-1\right) \sum_{k=0}^{m}c2^{k} \\
&\leq\binom{m+d-1}{d-1} e^{(c2^{m}/\lambda)^{\kappa}}\left(e^{(c2^{m}/\lambda)^{\kappa}}-1\right)c2^{m+1} \,.
\end{align*}
The time required to compute all $u_{i}$ with $\|i\|_{1}\leq n$ once is similarly bounded by
$$
C_{n}=\sum_{\|i\|_{1}\leq n}t_{i}
\leq\sum_{k=0}^{n}\binom{k+d-1}{d-1}t_{k}
\leq\binom{n+d-1}{d-1} c 2^{n+1}
$$ 
and hence $R_{m}/C_{n}\approx(m/n)^{d-1}2^{m-n}e^{(c 2^{m}/\lambda)^{\kappa}} (e^{(c
2^{m}/\lambda)^{\kappa}}-1)$ estimates the expected proportion of extra time spent on recomputations. We would generally expect that $c 2^{m}/\lambda<<1$
(we assume the time to compute level $m$ grids is much less than the mean time
to failure) and therefore this quantity is small. As an example, if we again consider a level $12$ computation in $3$
dimensions for which $t_{m}\leq 2^{-(m-12)}$ and the time to failure is Weibull distributed 
with mean $100$ seconds with shape parameter $0.7$, the expected proportion of time spent
recomputing solutions level $10$ or smaller is
$(10/12)^{2}2^{-2}e^{(0.25/100)^{0.7}}(e^{(0.25/100)^{0.7}}-1)\approx2.68\times10^{-3}$. In
comparison, if any of the $u_{i}$ which fail were to be recomputed then 
a proportion of $4.23\times10^{-2}$ additional time for would be expected for
recomputations, almost $16$ times more. Whilst this is a somewhat crude estimate it clearly demonstrates the our approach will scale better than a traditional checkpoint restart when faults are relatively frequent.

\section{Implementation and Scalability}\label{sec:implem}

At the heart of our implementation is a very simple procedure: solve the problem
on different grids, combine the solutions, and repeat.
This section is broken up into 
different sub-sections based upon where different layers of parallelism can be
implemented. 
We conclude by discussing some bottlenecks in the current implementation.

\subsection{Top Layer: Load Balancing and Combination}

The top layer is written in Python. As in many other applications, we use Python
to glue together the different components of our implementation as well as providing some high level functions for performing the combination. It is based
upon the development of NuMRF~\cite{iccs_paper,jay_parco}: intended to be
a clean interface where different computation codes can be easily
interchanged or added to the application. This layer can be further broken down
into 4 main parts.
The first is the loading of all dependencies including various Python and {\em
numpy} modules as well as any shared libraries that will be used to solve the
given problem. In particular, bottom layer components which have been compiled
into shared libraries from various languages (primarily C++ with C wrappers in
our case) are loaded into Python using {\em ctypes}.
The second part is the initialisation of data structures and
construction/allocation of arrays which will hold the relevant data. This is achieved
using PyGraFT~\cite{iccs_paper,jay_parco} which is a general class of grids and
fields that allows us to handle data from the various components in a generic
way. Also in this part of the code is the building of a sparse grid data
structure. This is done with our own C++ implementation which was loaded into
Python in the first part of the code.
The third part consists of solving the given problem. This is broken into
several "combination steps". A "combination step" consists of a series of time
steps of the underlying solver for each component solution, followed by a
combination of the component solutions into a sparse grid solution, and finally
a sampling of the component solutions from the sparse grid solution before
repeating the procedure. 
The fourth and final part of the code involves checking the error of the
computed solutions, reporting of various log data and finalise/cleanup.

The top layer is primarily responsible for the coarsest grain parallelism, that
is distributing the computation of different component solutions across
different processes. This is achieved through MPI using {\em mpi4py}\footnote{http://mpi4py.scipy.org/}. 
There are two main tasks the top layer must perform in order
to effectively handle this. The first is to determine an appropriate load
balancing of the different component solutions across a given number of
processes. For our simple problem this can be done statically on startup before 
initialising any data structures. For more
complex problems this can be done dynamically by evaluating the load balancing
and redistributing if necessary at start of each combination step based upon timings 
performed in the last combination step.
Re-distribution of the grids may require reallocating many of the data structures. 
The second task the top layer is responsible for is the communication between
MPI processes during the combination step. This is achieved using two {\sc
all{\_}reduce} calls. The first call is to establish which solutions have been
successfully computed. This is required so that all processes are able to
compute the correct combination coefficients. Each process then does a partial sum of
the component solutions it has computed. The second {\sc all{\_}reduce} call
then completes the combination of all component solutions distributing the
result to all processes. Following this the component solutions are then sampled
from the complete sparse grid solution.

\subsection{Bottom Layer: Solver and Sparse Grid Algorithms}

The bottom layer is made up of several different components, many of which are
specific to the problem that is intended to be solved. When solving our
advection problem we have 2 main components, one is responsible for the sparse
grid data structure and functions relating to the sparse grid (e.g.
interpolation and sampling of component solutions) and the other component is
the advection solver itself.
Both the sparse grid and advection solver components use OpenMP to achieve a
fine grain level of parallelism. This is primarily achieved by distributing the
work of large for loops within the code across different threads. The for loops
have roughly constant time per iteration so the distribution of work amongst
threads is done statically.

\subsection{Optional Middle Layer: Domain Decompositions}

The middle layer is currently being developed into the programming model. It is
intended solely to handle various aspects relating to the computation of
component solutions where domain decompositions are added as a third layer of
parallelism. This will be achieved through an interface with a distributed array
class of the NuMRF/PyGraFT framework at the top layer. This layer will need to interface
with solver kernel from the bottom layer and then perform communication of data
across domain boundaries. The combination of solutions onto the sparse grid in
the top layer will also need to interface with this layer to handle the
communication of different domains between MPI processes. 
It is intended that most of this will be transparent to the user.

\subsection{Scalability bottlenecks}

Since interpolation of the sparse grid and the solver (and any other time
consuming operations) each benefit from load balancing with MPI and work sharing
with OpenMP, any major hurdles to scalability will be caused by the {\sc all{\_}reduce}
communication and any serial operations in the code (e.g. initialisation
routines). Ignoring initialisation parts of the code it becomes clear we need to
either reduce the size of the data which is communicated, or reduce the
frequency at which it is communicated. The first can be done if we apply some
compression to the data before communicating, i.e. we trade-off smaller
communications for additional CPU cycles. Another way is to recognise that it is
possible to do the {\sc all{\_}reduce} on a sparse grid of level $n-1$ if a partial
hierarchisation is done to the largest component grids. This doesn't improve the
rate in which the complexity grows but can at least reduce it by a constant.
Reducing the frequency of the {\sc all{\_}reduce} can be done by performing partial
combinations in place of full combinations for some proportion of the steps.
This trades off the time taken to combine with some accuracy of the
approximation. A partial combination is where a grid combines only with its
neighbouring grids.
However, the only way to really address the bottleneck caused by communication is 
to perform a full hierarchisation of the component grids~\cite{Phillips_parco_paper}. 
By doing this one can significantly reduce the communication volume at the expense 
of increasing the number of messages. One can then reduce the number of messages 
by identifying those which are communicated to the same {\sc mpi} processes. We 
currently have a first implementation of this which we intend to improve as 
development continues.

\section{Numerical Results}\label{sec:numres}

%

In this section, we present some numerical results which validate our approach. 
The problem used to test our algorithm is the scalar advection equation
$$
\frac{\partial u}{\partial t}+a \cdot \nabla u = 0
$$
on the domain $[0,1]^{3}\subset\mathbb{R}^{3}$ for constant $a\in\mathbb{R}^{3}$. 
For the results presented in this section we use $a=(1,1,1)$, periodic 
boundary conditions and the initial condition
$$
u_{0}(x)=\sin(4\pi x_{1})\sin(2\pi x_{2})\sin(2\pi x_{3}) \,.
$$
The PDE is solved using a Lax-Wendroff finite difference scheme giving results which are
second order in space and time. We compare numerical solutions against the exact solution
$$
u(x,t)=\sin(4\pi (x_{1}-a_{1}t))\sin(2\pi (x_{2}-a_{2}t))\sin(2\pi (x_{3}-a_{3}t))
$$
to determine the solution error at the end of each computation.

A truncated combination technique as in~\eqref{eqn:tct} is used for our experiments. 
In order to apply the FTCT we need to compute some additional grids. We define
$$
I_{n,\tau}:=\{i\in\mathbb{N}^{d}:\min(i)\geq \tau\text{ and }n-d-1\leq\|i\|_{1}\leq
n\}
$$ 
which is the set of indices for which we are required to compute solutions
$u_{i}$ if the top two levels are not to be recomputed in the event of a fault.

Note that as the grid sizes vary between the $u_{i}$ so does the maximum stable time 
step size as determined by the CFL condition. 
We choose the same time step size for all component solutions to avoid 
instability that may otherwise arise from the extrapolation of time stepping  
errors during the combination. As a result our timesteps must satisfy the CFL 
condition for all component grids. By choosing $\Delta t$ such that it satisfies 
the CFL condition for the numerical solution of $u_{(n-2\tau,n-2\tau,n-2\tau)}$ 
it follows that the CFL condition is also satisfied for all $u_{i}$ with $i\in I_{n,\tau}$.

All of our computations were performed on a Fujitsu PRIMERGY cluster consisting 
of 36 nodes each with 2 Intel Xeon X5670 CPUs (6 core, 2.934GHz) with Infiniband 
interconnect.

\begin{table}
\centering
\caption{\label{tab:wr2} 
Numerical results for $r=100$ runs for each $l,\tau$ using the Weibull distribution with mean of $1000$ seconds and shape parameter of $0.7$ for the fault simulation. The computation was performed on 2 nodes with 6 OpenMP threads on each.
}
\begin{tabular}{ | c | c | c | c | c | c | c | c | c | }
\hline
$l$ & $\tau$ & $f_{ave}$ & $\epsilon_{ave}$ & $\epsilon_{min}$ & $\epsilon_{max}$ & $w_{ave}$ & $w_{min}$ & $w_{max}$ \\
\hline
18 & 4 & 0.13 & 2.103e-4 & 2.096e-4 & 2.298e-4 & 26.96 & 26.94 & 27.07 \\
20 & 5 & 0.37 & 7.064e-5 & 7.004e-5 & 8.580e-5 & 72.03 & 71.82 & 72.72 \\
19 & 4 & 0.48 & 6.355e-5 & 6.266e-5 & 6.899e-5 & 131.1 & 130.8 & 131.9 \\
21 & 5 & 1.21 & 1.979e-5 & 1.886e-5 & 3.030e-5 & 379.9 & 379.3 & 381.9 \\
20 & 4 & 1.80 & 1.925e-5 & 1.856e-5 & 2.156e-5 & 649.8 & 648.2 & 653.3 \\
\hline
\end{tabular}
\end{table}

\begin{table}
\centering
\caption{\label{tab:wr6} 
Numerical results for $r=200$ runs for each $l,\tau$ using the Weibull distribution with mean of $1000$ seconds and shape parameter of $0.7$ for the fault simulation. The computation was performed on 6 nodes with 6 OpenMP threads on each.
}
\begin{tabular}{ | c | c | c | c | c | c | c | c | c | }
\hline
$l$ & $\tau$ & $f_{ave}$ & $\epsilon_{ave}$ & $\epsilon_{min}$ & $\epsilon_{max}$ & $w_{ave}$ & $w_{min}$ & $w_{max}$ \\
\hline
18 & 4 & 0.305 & 2.119e-4 & 2.096e-4 & 2.418e-4 & 9.216 & 9.196 & 9.392 \\
20 & 5 & 0.535 & 7.166e-5 & 7.004e-5 & 1.170e-4 & 24.77 & 24.67 & 25.22 \\
19 & 4 & 0.690 & 6.385e-5 & 6.244e-5 & 7.392e-5 & 44.55 & 44.40 & 45.90 \\
21 & 5 & 1.805 & 2.015e-5 & 1.886e-5 & 3.179e-5 & 131.7 & 130.9 & 134.6 \\
20 & 4 & 2.475 & 1.961e-5 & 1.844e-5 & 2.478e-5 & 224.4 & 223.3 & 228.7 \\
\hline
\end{tabular}
\end{table}

\subsection{Solution Error} 
We first looked at the effect of simulated faults on the error the computed solution. 
Given level $n$ and truncation parameter $\tau$ the code was executed for some number 
of runs $r$ on a fixed number of nodes using the same number of threads. Component 
solutions are combined twice in each run, once halfway through the time steps and 
again at the end. 
For each run we recorded the number of faults $f$ that occurred, the $l_{1}$ error of 
the solution $\epsilon$ and the wall time $w$ spent in the solver. 
We then calculated the average number of faults
$$
f_{ave}=\frac{1}{r}\sum_{k=1}^{r}f_{k} \,,
$$ 
the average, minimal and maximal observed errors
$$
\epsilon_{ave}=\frac{1}{r}\sum_{k=1}^{r}\epsilon_{k} \,,\quad \epsilon_{min}=\min\{\epsilon_{1},\dots,\epsilon_{r}\} \,,\quad \epsilon_{max}=\max\{\epsilon_{1},\dots,\epsilon_{r}\} \,,
$$
and the average, minimal and maximal observed wall times
$$
w_{ave}=\frac{1}{r}\sum_{k=1}^{r}w_{k} \,,\quad w_{min}=\min\{w_{1},\dots,w_{r}\} \,,\quad w_{max}=\max\{w_{1},\dots,w_{r}\} \,.
$$

Table~\ref{tab:wr2} shows our results for $r=100$ runs of the FTCT with fault simulation 
on 2 nodes with 6 OpenMP threads on each. Faults were simulated as described 
in Section~\ref{sec:fs} using the Weibull distribution with mean of $1000$ seconds 
and shape parameter $0.7$ to sample the time between failures. As we increase the 
level $n$ (or decrease $\tau$) we increase the problem size and hence computation 
time. This in turn leads to an increase in the average number of faults that occur 
per run as seen in the $f_{ave}$ column. The minimal error is the same as the 
error without failure (sometimes it is fractionally smaller). Comparing with the 
average error we see that the additional error generated by recovery from simulated 
faults is small. Also worth noting is that the variability in computation time is quite
small indicating that any recomputations, when they occur, do not seem to cause any
significant disruptions.

\begin{table}
\centering
\caption{\label{tab:er2} 
Numerical results for $r=100$ runs for each $l,\tau$ using the exponential distribution with mean of $1000$ seconds for the fault simulation. The computationa were performed on 2 nodes with 6 OpenMP threads on each.
}
\begin{tabular}{ | c | c | c | c | c | c | c | c | c | }
\hline
$l$ & $\tau$ & $f_{ave}$ & $\epsilon_{ave}$ & $\epsilon_{min}$ & $\epsilon_{max}$ & $w_{ave}$ & $w_{min}$ & $w_{max}$ \\
\hline
18 & 4 & 0.06 & 2.098e-4 & 2.096e-4 & 2.230e-4 & 26.96 & 26.93 & 27.10 \\
20 & 5 & 0.17 & 7.098e-5 & 7.006e-5 & 1.155e-4 & 72.01 & 71.87 & 72.47 \\
19 & 4 & 0.26 & 6.321e-5 & 6.266e-5 & 7.283e-5 & 131.1 & 130.8 & 131.7 \\
21 & 5 & 0.71 & 1.942e-5 & 1.886e-5 & 2.980e-5 & 379.8 & 379.3 & 380.8 \\
20 & 4 & 1.27 & 1.921e-5 & 1.856e-5 & 2.127e-5 & 649.3 & 647.9 & 653.4 \\
\hline
\end{tabular}
\end{table}

\begin{table}
\centering
\caption{\label{tab:er6} 
Numerical results for $r=100$ runs for each $l,\tau$ using the exponential distribution with mean of $1000$ seconds for the fault simulation. The computations were performed on 6 nodes with 6 OpenMP threads on each.
}
\begin{tabular}{ | c | c | c | c | c | c | c | c | c | }
\hline
$l$ & $\tau$ & $f_{ave}$ & $\epsilon_{ave}$ & $\epsilon_{min}$ & $\epsilon_{max}$ & $w_{ave}$ & $w_{min}$ & $w_{max}$ \\
\hline
18 & 4 & 0.070 & 2.103e-4 & 2.096e-4 & 2.296e-4 & 9.231 & 9.214 & 9.340 \\
20 & 5 & 0.155 & 7.034e-5 & 7.004e-5 & 8.162e-5 & 24.80 & 24.70 & 30.83 \\
19 & 4 & 0.265 & 6.327e-5 & 6.244e-5 & 7.039e-5 & 44.48 & 44.36 & 45.07 \\
21 & 5 & 0.865 & 1.936e-5 & 1.886e-5 & 3.203e-5 & 131.7 & 131.0 & 134.5 \\
20 & 4 & 1.415 & 1.921e-5 & 1.844e-5 & 2.178e-5 & 224.0 & 222.9 & 227.4 \\
\hline
\end{tabular}
\end{table}

In Table~\ref{tab:wr6} we repeat this experiment with $r=200$ runs on 6 nodes 
with 6 OpenMP threads on each. Whilst running with additional nodes leads to a 
decrease in computation time we experience more faults on average because of 
the additional nodes. However, we can see that the effect of the increased average 
number of faults is quite small on both the average solution error and the average 
wall time.

Table~\ref{tab:er2} again shows results for $r=100$ runs of the FTCT with fault simulation
on 2 nodes with 6 OpenMP threads on each. However, for this experiment the faults are 
exponentially distributed with a mean of $1000$ seconds. We see that for this
distribution the faults are a little less frequent on average leading to a
slightly smaller average error. Similar is observed in Table~\ref{tab:er6} where we
repeat the experiment with $r=200$ runs on 6 nodes with 6 OpenMP threads on each. 
Here the average number of faults is substantially less than the results of 
Table~\ref{tab:wr6} and this is again reflected by a smaller average error in comparison. 
The large $w_{max}$ in the 2nd row is due to a single outlier, the next largest time being 
$25.34$. No simulated faults occurred for this outlier so we suspect it was due to a system issue.

\subsection{Scalability}

In Figure~\ref{fig:scala2} we demonstrate the scalability and efficiency of our
implementation when the fault simulation is disabled. Noting from 
Table~\ref{tab:wr6} that faults have very little effect on the computation time 
we expect similar results with fault simulation turned on. 
The advection problem was solved using a $n=22, \tau=6$ truncated combination. 
The component solutions were combined only once at the end of the computation.
The {\em solver} time reported here is the timing of the core of the code, 
that is the repeated computation, combination and communication of the solution 
which is highly scalable. The {\em total} time reported here includes the time
for Python to load modules and shared libraries, memory allocation and error checking. 
The error checking included in the {\em total} time is
currently computed in serial and could benefit from OpenMP parallelism.

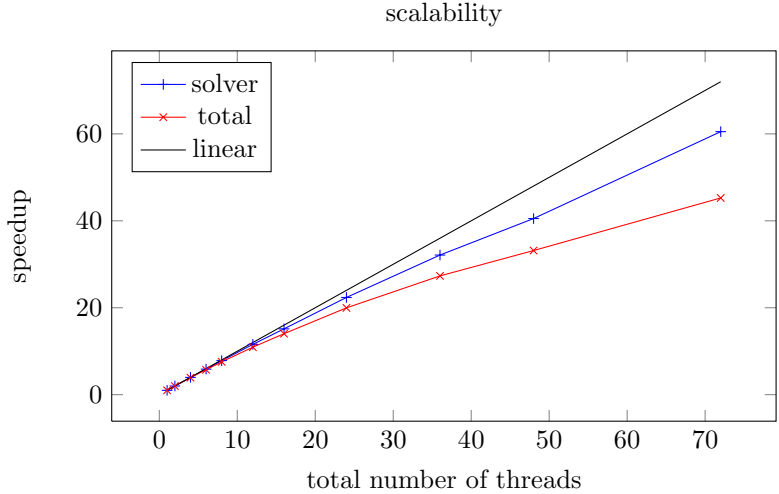
\begin{figure}
\centering
\begin{tikzpicture}
\begin{axis}[width=0.8\textwidth,
	height=0.5\textwidth,
	ylabel=speedup,
	xlabel=total number of threads,
	title=scalability,
	legend style={legend pos=north west}
]
\addplot[color=blue,mark=+]
	coordinates {(1, 1.0) (2, 1.999503001868713) (4, 3.9778915936641024) (6, 5.878718802969197) (8, 7.836917562724014) (12, 11.574774152713045) (16, 15.145158861617226) (24, 22.36579942183678) (36, 32.12360268284893) (48, 40.523368251410155) (72, 60.516847172081825)};
\addlegendentry{solver}
\addplot[color=red,mark=x]
	coordinates {(1, 1.0) (2, 1.9895584696232995) (4, 3.91569199837609) (6, 5.72765319684416) (8, 7.554623102457946) (12, 10.928991528624614) (16, 14.031659161759613) (24, 19.969141279700285) (36, 27.32690232056125) (48, 33.1609364767518) (72, 45.26189944134078)};
\addlegendentry{total}
\addplot[color=black,mark=none]
	coordinates {(1, 1.0) (2, 2.0) (4, 4.0) (6, 6.0) (8, 8.0) (12, 12.0) (16, 16.0) (24, 24.0) (36, 36.0) (48, 48.0) (72, 72.0)};
\addlegendentry{linear}
\end{axis}
\end{tikzpicture}
\caption{\label{fig:scala2} This plot demonstrates the scalability of our implementation for the advection problem using a $l=22, \tau=6$ truncated combination. Fault simulation was disabled and only one combination was performed at the end of all computations. 
}
\end{figure}

In Figure~\ref{fig:tvf} we compare the computation time required for our 
approach to reach a solution compared to more traditional checkpointing approaches, 
in particular, with a local and global checkpointing approach. 
With global checkpointing we keep a copy of the last combined solution. 
If a failure affects any of the component grids it is assumed that the entire 
application is killed and computations must be restarted from the most recent combined solution. 
We emulate this by checking for faults at each combination step and restart from the last 
combination step if any faults have occurred. 
With local checkpointing each MPI process saves a copy of each component solution it computes. 
In this case when a faults affect component solutions we need only recompute the affected 
component solutions from their saved state. 
In both checkpointing methods the extra component solutions 
used in our approach are not required and are hence not computed. 
As a result these approaches are slightly faster when no faults occur. 
However, as the number of faults increases, it can be seen from Figure~\ref{fig:tvf} 
that the computation time for the local and global checkpointing methods begins 
to grow. A line of best fit has been added to the figure which makes it clear that the 
time for recovery with global checkpointing increases rapidly with the number of 
faults. Local checkpointing is a significant improvement on this but still shows some growth. 
On the other hand our approach is barely affected by the number of faults and 
beats both the local and global checkpointing approaches after only a few faults. 
For much larger number of faults our approach is significantly better.

\begin{figure}
\centering
\begin{tikzpicture}
	\begin{axis}[width=0.8\textwidth,
	height=0.5\textwidth,
	ylabel=total computation time,
	xlabel=total number of faults,
	title=time vs. faults for 3D advection,
	legend style={legend pos=north east},
	xmin=-5,
	xmax=55,
	ymin=120,
	ymax=410, 
	legend entries={recombine,
				  local checkpoint,
				  global checkpoint}
	]
	\addlegendimage{mark=x,blue}
	\addlegendimage{mark=+,red}
	\addlegendimage{mark=o,black}
	\pgfplotstableread{rc_scatter.dat}\tableA
	\addplot[mark=x,only marks,color=blue] table[x=f,y=t] from \tableA;
	\pgfplotstableread{rc_bestfit.dat}\tableB
	\addplot[no marks,color=blue,forget plot] table[x=f,y=t] from \tableB;
	\pgfplotstableread{cp_scatter.dat}\tableC
	\addplot[mark=+,only marks,color=red] table[x=f,y=t] from \tableC; 
	\pgfplotstableread{cp_bestfit.dat}\tableD
	\addplot[no marks,color=red] table[x=f,y=t,forget plot] from \tableD;
	\pgfplotstableread{gcp_scatter.dat}\tableE
	\addplot[mark=o,only marks,color=black] table[x=f,y=t] from \tableE;
	\pgfplotstableread{gcp_bestfit.dat}\tableF
	\addplot[no marks,color=black] table[x=f,y=t,forget plot] from \tableF;
	\end{axis}		
\end{tikzpicture}
\caption{\label{fig:tvf} We compare the time taken to compute the solution to the 3D advection problem using three different approaches to fault tolerance. 
The problem size is fixed at level 21 with truncation parameter 5. 
All computations used 6 MPI processes with 6 OpenMP threads each. 
Component solutions are combined 4 times throughout the computation and it is during the combination that we check for faults. 
The {\emph recombine} method is our approach described in Section~\ref{sec:ftct}. 
The {\emph local checkpoint} method involves each {\sc mpi} process checkpointing component solutions. 
The {\emph global checkpoint} method involves all {\sc mpi} processes checkpointing the last combined solution. 
For each method the problem was run numerous times with {\sc mttf} varying from $25$ to $1000$ seconds.
}
\end{figure}
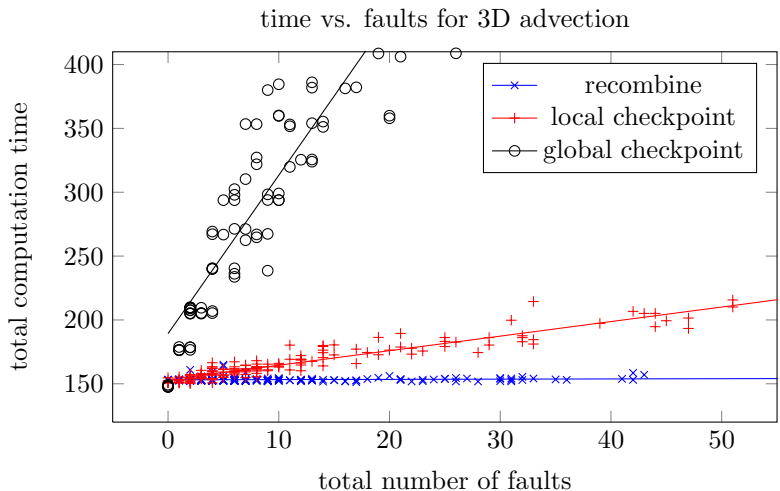

\section*{Conclusion}


A generalisation of the sparse grid combination technique has been presented. 
From this generalisation a fault tolerant combination technique has been proposed which 
significantly reduces recovery times at the expense of some upfront overhead 
and reduced solution accuracy. Theoretical bounds on the expected error and 
numerical experiments show that the reduction in solution accuracy is very small. 
The numerical experiments also demonstrate that the upfront overheads become negligible
compared to the costs of recovery using checkpoint-restart techniques if several faults occur.
There are some challenges associated with load balancing and efficient communication
with the implementation of the combination technique. Studying these aspects and 
improving the overall scalability of the initial implementation will be the subject of future work. 
As the ULFM specification continues to develop, the validation of the FTCT on a system with
real faults is also being investigated.

\end{document}